\documentclass[12pt]{amsart}

\usepackage{amssymb,latexsym,amscd,graphicx}
\usepackage[pdfauthor={David Cook II, Juan Migliore, Uwe Nagel, and Fabrizio Zanello},
            pdftitle={An algebraic approach to finite projective spaces},
            pdfsubject={Algebraic combinatorics},
            pdfkeywords={Finite projective plane; linear space; Weak Lefschetz Property; Strong Lefschetz Property; minimal free resolution; monomial algebra; level algebra; Stanley-Reisner ring; inverse system.},
            pdfproducer={LaTeX},
            pdfcreator={latex},
            pdfborder={0 0 0},
            colorlinks=true]{hyperref}
\usepackage{color}

\usepackage[margin=1in]{geometry}

\newtheorem{thm}{Theorem}[section]
\newtheorem{prop}[thm]{Proposition}
\newtheorem{lem}[thm]{Lemma}
\newtheorem{cor}[thm]{Corollary}
\newtheorem{conj}[thm]{Conjecture}

\theoremstyle{definition}
\newtheorem{defn}[thm]{Definition}

\newtheorem{example}[thm]{Example}
\newtheorem{rem}[thm]{Remark}

\newtheorem{lemma}[thm]{Lemma}

\theoremstyle{remark}

\def\cocoa{{\hbox{\rm C\kern-.13em o\kern-.07em C\kern-.13em o\kern-.15em A}}}

\newcommand{\skipit}[1]{{}}
\newcommand{\prfend}{\hbox to7pt{\hfil}%
    \par\vskip-\baselineskip\hbox to\hsize%
    {\hfil\vbox {\hrule width6pt height6pt}}\vskip\baselineskip}

\newcommand{\st}{\ensuremath{\colon\,}}                       
\newcommand{\ZZ}{\ensuremath{\mathbb{Z}}}                     
\newcommand{\fm}{\mathfrak m}

\DeclareMathOperator{\Tor}{Tor}
\DeclareMathOperator{\depth}{depth}

\DeclareMathOperator{\Char}{char}
\DeclareMathOperator{\link}{link}                             

\newcount\HOUR
\newcount\MINUTE
\newcount\HOURSINMINUTES
\newcount\INTVAL
\newcommand{\twodigit}[1]{\INTVAL=#1\relax\ifnum\INTVAL<10 0\fi\the\INTVAL}
\HOUR=\time\divide\HOUR by 60\relax
\HOURSINMINUTES=\HOUR\multiply\HOURSINMINUTES by 60\relax
\MINUTE=\time\advance\MINUTE by -\HOURSINMINUTES\relax
\newcommand\rightnow{
    \twodigit{\the\HOUR}:\twodigit{\the\MINUTE},
    \twodigit{\number\day}.\space
    \ifcase\month\or January\or February\or March\or April\or May\or June\or July\or August\or September\or October\or November\or December\fi
    \space\number\year}

\begin{document}


\title[An algebraic approach to finite projective planes]{An algebraic approach to finite projective planes}

\author[D.\ Cook II]{David Cook II}
\address{Department of Mathematics \& Computer Science, Eastern Illinois University, Charleston, IL 61920}
\email{\href{mailto:dwcook@eiu.edu}{dwcook@eiu.edu}}

\author[J.\ Migliore]{Juan Migliore}
\address{Department of Mathematics, University of Notre Dame, Notre Dame, IN 46556}
\email{\href{mailto:migliore.1@nd.edu}{migliore.1@nd.edu}}

\author[U.\ Nagel]{Uwe Nagel}
\address{Department of Mathematics, University of Kentucky, Lexington, KY 40506}
\email{\href{mailto:uwe.nagel@uky.edu}{uwe.nagel@uky.edu}}

\author[F.\ Zanello]{Fabrizio Zanello}
\address{Department of Mathematical Sciences, Michigan Technological University, Houghton, MI 49931}
\email{\href{mailto:zanello@mtu.edu}{zanello@mtu.edu}}

\begin{abstract}
A finite projective plane, or more generally a finite linear space, has an associated incidence complex that gives rise to two natural algebras: the Stanley-Reisner ring $R/I_\Lambda$ and the inverse system algebra $R/I_\Delta$. We give a careful study of both of these algebras. Our main results are a full description of the graded Betti numbers of both algebras in the more general setting of linear spaces (giving the result for the projective planes as a special case), and a classification of the characteristics in which the inverse system algebra associated to a finite projective plane has the Weak or Strong Lefschetz Property.
\end{abstract}

\thanks{The work for this paper was done while J.\ Migliore was partially supported by the National Security Agency under Grant Number H98230-12-1-0204 and by a Simons Foundation grant (\#309556), 
while U. Nagel was partially supported by the National Security Agency under Grant
Number H98230-12-1-0247 and by the Simons Foundation under grant \#317096,  and while F.\ Zanello was partially supported by a Simons Foundation grant (\#274577).}

\maketitle


\section{Introduction}\label{sec:intro}

The purpose of this note is to introduce and begin discussing a new connection between commutative algebra and finite geometries, especially finite projective planes. Namely, given a projective plane of order $q$ (or more generally, a linear space), we will naturally associate two monomial algebras to it: one that comes from Macaulay's inverse systems, and the other from Stanley-Reisner theory, by viewing the plane as a simplicial complex. 

We refer the reader to standard texts such as \cite{MS,St} for the main definitions and facts of combinatorial commutative algebra, and to \cite{Ge,IK} for information on inverse systems, though we will only employ this latter theory in the context of monomial algebras, where it is much simpler. We refer to Moorhouse's course notes \cite{Mo} to recall the following geometric definitions.

\begin{defn}
A \emph{point-line incidence structure} is a pair $(P, L)$, where $P$ is a finite set of points and  $L$ a finite set of lines, equipped with a binary relation $I \subset P \times L$ such that $(p, \ell) \in I$ precisely when $p$ lies on $\ell$ (that is, $p$ is \emph{incident} to $\ell$). 

A point-line incidence structure $(P, L)$ is a \emph{linear space} if any two distinct points lie on exactly one common line and any line contains at least two points. 

If every line in $L$ has the same number of points, we say that $(P, L)$ is \emph{equipointed}; otherwise, $(P, L)$ is \emph{nonequipointed}.
\end{defn}

A finite projective plane is then a special case of a linear space.

\begin{defn}
A \emph{finite projective plane} is a linear space $(P,L)$ such that the following extra conditions are satisfied:
\begin{enumerate}
\item Any two distinct lines meet at (exactly) one point;
\item There exist three noncollinear points;
\item Every line contains at least three points.
\end{enumerate}
\end{defn}

It is well known that, for every finite projective plane $(P, L)$, there exists an integer $q$ such that
$\#P = \#L = q^2 + q + 1$, every point lies on exactly $q+1$ lines, and every line contains 
exactly $q+1$ points.  In this case, $(P, L)$ is said to have \emph{order} $q$.

Our paper is structured as follows. Extending a construction for projective planes, in the next section we introduce the incidence complex of a linear space $(P,L)$ and study its  Stanley-Reisner ring $R/I_{\Lambda}$, with an eye on the most interesting case of a finite projective plane. The main fact shown in Section \ref{sec:sr} is a complete characterization of the graded Betti numbers (a much stronger set of invariants than the Hilbert function) for any Stanley-Reisner ring $R/I_{\Lambda}$ (see Theorem \ref{thm:Betti lin space}). In particular, it will follow that, except in trivial situations, all of these rings have depth 2, and are therefore very far from being Cohen-Macaulay.

In Section \ref{sec:is} we  present our second algebraic approach, by associating to any given linear space $(P,L)$ a  naturally defined artinian algebra $R/I_{\Delta}$, which is constructed via (monomial) inverse systems. Since the ideal $I_{\Delta}$ can alternatively be obtained from the Stanley-Reisner ideal  $I_{\Lambda}$ by adding the squares of all the variables to the generating set, we  in part rely on the results of Section \ref{sec:sr} to develop this approach. Also for the algebras $R/I_{\Delta}$, we are able to provide a full characterization of the graded Betti numbers (see Theorem \ref{thm:Betti w squares}). We notice that, once again, our results carry a very neat statement when specialized to projective planes (see Corollary \ref{cor:Betti proj plane w squares}).

In Section \ref{sec:lefschetz}, we then turn to the \emph{Weak Lefschetz Property} (WLP) and the \emph{Strong Lefschetz Property} (SLP). Recall that an artinian $K$-algebra $A= \oplus_{i=0}^e A_i$ is said to have the WLP if, for a \emph{general} (according to the Zariski topology) linear form $\ell$, all of the multiplication maps $\times \ell$ between the $K$-vector spaces $A_i$ and $A_{i+1}$ have maximal rank (i.e., each is injective or surjective). Similarly, we say that $A$ has the SLP if, for all $i, d\geq 0$, the maps $\times \ell^d$ between $A_i$ and $A_{i+d}$ have maximal rank. 

The Lefschetz Properties, whose study was introduced by Richard Stanley in his work in combinatorial commutative algebra in the Seventies, can be seen as an algebraic abstraction of the Hard Lefschetz theorem of algebraic geometry, and their existence carries several important consequences (e.g., to name one in the combinatorial direction, if an algebra $A$ has the WLP, then its Hilbert function is unimodal; see \cite{HMNW} for basic facts on algebras with the WLP and SLP). In fact, a substantial amount of research has been done in recent years on the WLP and the SLP for \emph{monomial} algebras, which are the main object of this paper, and much of this work has been motivated by the surprising connections that have emerged with combinatorics, in particular with plane partitions and lattice paths. See, as a nonexhaustive list, \cite{CGJL, cook, 3, CN2, KV, LZ, MMN-mon}.

Notice that the existence of the Lefschetz Properties for any algebra $A$ of positive depth (over an infinite field) is a trivial problem, because $A$ is always guaranteed to have a linear nonzero divisor, which immediately implies the injectivity of all of the multiplication maps between its graded components. Hence, both Lefschetz Properties hold for the Stanley-Reisner ring $R/I_{\Lambda}$ associated to any linear spaces $(P, L)$. However, investigating the WLP and the SLP for the inverse system artinian algebras $R/I_{\Delta}$ turns out to be an interesting and highly nontrivial problem. This is the object of Section \ref{sec:lefschetz},  where we give a careful study of the characteristics of the base fields over which the algebra $R/I_{\Delta}$ has the WLP in both the equipointed and the nonequipointed cases. For the equipointed case we also study the SLP, and 
the main application  will again be to the case of finite projective planes. In particular, we will provide a classification of the characteristics of the base fields over which the algebra $R/I_{\Delta}$ corresponding to a projective plane has, respectively, the WLP and  the SLP (see Theorems \ref{main} and \ref{almost-equipointed}).

As we mentioned earlier, our chief goal in this paper is to begin a study of some interesting and potentially fruitful connections between commutative algebra and finite geometries. In general, the relationship between these two mathematical areas has only marginally been explored so far, and both the nature of the approaches outlined earlier and the neat statements of some of the results of the next sections strongly suggest that much more can be done in this line of research. In particular, we have only commenced to investigate the potential impact that certain algebraic tools  might have in studying problems of finite geometries. For example, a fundamental  problem in the theory of finite projective planes is to characterize the integers $q$ that may occur as orders of the planes.  Thus, one of the most natural and consequential questions that the interested reader may want to try to address in a subsequent work is: is it possible to use combinatorial commutative algebra to impose new, nontrivial restrictions on the possible values of $q$? It is widely believed that $q$ can only be a power of a prime number, but even the existence of a projective plane of order $12$ is still open. Perhaps the fact that the graded Betti numbers of the incidence complex of a projective plane are determined by the order of the plane  (see Corollary \ref{cor:Betti proj plane}) can serve as a starting point  for future investigations. 


\section{The Stanley-Reisner ring associated to a linear space}\label{sec:sr}

We begin by describing an incidence complex associated to a linear space.  We will prove that the Alexander dual of the incidence complex is vertex-decomposable. Therefore the Stanley-Reisner ring associated to the incidence complex has a $3$-linear resolution and has depth $2$.  Then, we will describe the graded Betti numbers of the Stanley-Reisner ring.
	
For the main definitions and some basic results on Stanley-Reisner theory (and monomial ideals in general), we refer to \cite{HH,MS,St}.

In order to study linear spaces, we define a simplicial complex, the incidence complex.

\begin{defn}
    The \emph{incidence complex} of a linear space $(P, L)$ is the simplicial complex
    $\Lambda$ on $P$ with facets given by the collection of points on each line in $L$.
\end{defn}

Clearly $\Lambda$ is \emph{pure} (i.e. the facets all have the same dimension) if and only if $(P, L)$ is equipointed. We make some further comments about the incidence complex.

\begin{prop}\label{pro:icprop}
    Let $\Lambda$ be the incidence complex of a linear space $(P, L)$.
    \begin{enumerate}
        \item The $f$-vector of $\Lambda$ is of the form
            \[
                f(\Lambda) = \left(1, \#P, \binom{\#P}{2}, \sum_{\ell \in L} \binom{\#\ell}{3}, \ldots, \sum_{\ell \in L} \binom{\#\ell}{t} \right)
            \]
            where $t = \max\{\#\ell \st \ell \in L\}$.
        \item The minimal nonfaces of $\Lambda$ are precisely the noncollinear triples of points.
        \item For $p \in P$, $\Lambda \setminus p$ is the incidence complex of $(P \setminus \{p\}, L')$, where $L'$
            is the set of lines in $L$ not containing $p$ together with the lines in $L$  containing at least three points,
            one of which is $p$, with $p$ removed.
        \item For $p \in P$, $\link_{\Lambda}(p)$ is a disjoint union of simplices.
    \end{enumerate}
\end{prop}
\begin{proof}
    The first three are obvious.  The fourth follows since every pair of points lies on a unique line.  Thus the lines containing $p$
    partition the points of $P \setminus \{p\}$.  The parts of the partition, being on a common line, thus form the simplices of the link.
\end{proof}

Recall that the \emph{Alexander dual} of a simplicial complex $\Delta$ is the simplicial complex $\Delta^{\vee}$ of complements of
nonfaces of $\Delta$ (see, e.g., \cite[Section~1.5.3]{HH}).

\begin{cor}
    Let $\Lambda$ be the incidence complex of a linear space $(P, L)$, where $\#L>1$. The Alexander dual
    $\Lambda^{\vee}$ is pure and $(\#P - 4)$-dimensional.
\end{cor}

Furthermore, Eagon and Reiner~\cite[Theorem~3]{ER} showed that  the Alexander dual $\Delta^{\vee}$ of the simplicial complex is Cohen-Macaulay
over the field $K$ if and only if the Stanley-Reisner ring $K[\Delta]$ has a linear resolution.  In the stronger situation that $\Delta^{\vee}$ is 
pure and shellable, Eagon and Reiner~\cite[Corollary~5]{ER} also showed that  $K[\Delta]$ has a linear resolution regardless of
the field $K$.

Following~\cite[Definition~2.1]{PB}, a pure complex $\Delta$ is \emph{vertex-decomposable} if either $\Delta$ is a simplex (including the void complex)
or there exists a vertex $v \in \Delta$, called a \emph{shedding vertex}, such that both $\link_{\Delta}{v}$ and $\Delta \setminus v$ are vertex-decomposable.
It is easy to see that vertex-decomposability implies shellability.

We state without proof a simple lemma about Alexander duality and subcomplexes.

\begin{lem}
    If $v$ is any vertex of a simplicial complex $\Delta$, then
    \[
        (\Delta \setminus v)^{\vee} = \link_{\Delta^{\vee}} v \text{~and~} (\link_{\Delta} v)^{\vee} = \Delta^{\vee} \setminus v.
    \]
\end{lem}

\begin{thm}
    Let $\Lambda$ be the incidence complex of a linear space $(P, L)$. Then the Alexander dual $\Lambda^{\vee}$ is vertex-decomposable.
\end{thm}
\begin{proof}
    We proceed by induction on $\#P$.  If $\#P = 2$, then $\Lambda$ is a simplex and so $\Lambda^{\vee}$ is the void complex, 
    which is vertex-decomposable.

    Suppose $\#P > 2$, and let $p \in P$.  It is easy to see that $\Lambda^{\vee} \setminus p$ is vertex-decomposable, since it is the
    Alexander dual of the disjoint union of simplices $\link_{\Lambda} p$.  On the other hand, $\link_{\Lambda^{\vee}} p$ is the 
    Alexander dual of an incidence complex of  a linear space on $\#P-1$ vertices by Proposition~\ref{pro:icprop}(iii),
		which is vertex-decomposable by induction.    
\end{proof}

\begin{cor}
     \label{cor:3-lin}
    Let $\Lambda$ be the incidence complex of a linear space $(P, L)$.  Then the Stanley-Reisner ring $K[\Lambda]$ has
    a $3$-linear resolution.
\end{cor}
\begin{proof}
    Since $\Lambda^{\vee}$ is vertex-decomposable and thus shellable, $K[\Lambda]$ has a linear resolution by~\cite[Corollary~5]{ER} 
    (see also~\cite[Theorem~8.1.9]{HH}).  Moreover, since the minimal nonfaces of $\Lambda$ are triples, it is thus a $3$-linear
    resolution.
\end{proof} 

In order to apply this fact we need a result on graded Betti numbers. Recall that graded Betti numbers of a simplicial complex $\Lambda$ on $d$ vertices are defined as 
\[
    \beta_{i, j} (K[\Lambda]) = \dim_K [\Tor^i_R (K, K[\Lambda])]_j, 
\]
where $R = K[x_1,\ldots,x_d]$. Its $i$-th total Betti number is 
\[
    \beta_{i} (K[\Lambda]) = \dim_K \Tor^i_R (K, K[\Lambda]). 
\]
The depth of a graded $K$-algebra $A$ is the maximum length of a regular sequence in $A$. If $A$ is a quotient of $R$, then its depth is $d - u$, where $u$ is the length of a minimal free resolution of $A$ over $R$. 

\begin{prop}
    \label{prop:betti e-linear} 
    Let $\Lambda$ be a simplicial complex  on $d$ vertices such that its Stanley-Reisner ideal has an $e$-linear resolution. Then, for each integer $i \ge 1$, 
    \[
        \beta_{i} (K[\Lambda]) = \beta_{i, i+e-1} (K[\Lambda]) = \sum_{j = 0}^{i+e-1} (-1)^{e-1+j} f_{j-1} (\Lambda) \binom{d-j}{i+e-1-j}. 
    \]

    In particular, the depth of $k[\Lambda]$ is at least $e-1$. 
\end{prop} 

\begin{proof}
    In order to simplify notation, set 
    \[
        \beta_i = \beta_{i, i+e-1} (K[\Lambda]) \quad \text{ and } \quad f_j = f_{j} (\Lambda). 
    \]
    
    By~\cite[Proposition~6.2.1]{HH}, the Hilbert series of $K[\Lambda]$ is given by
    \[
        H_{K[\Lambda]} (t) =  \sum_{i=0}^{m} f_{i-1} \, t^i \, (1-t)^{-i}, 
    \]
    where $m - 1 = \dim \Lambda$. 
       
    By assumption, the minimal free resolution of $K[\Lambda]$ has the form
    \[
        0 \to R^{\beta_d} (-e-d+1) \to \cdots \to R^{\beta_2} (-e-1) \to R^{\beta_1} (-e) \to R \to K[\Lambda] \to 0.
    \]
    
    Thus, its Hilbert series can be rewritten as 
    \[
        H_{K[\Lambda]} (t) = \frac{1}{(1-t)^d} \left [1 +  \sum_{i \ge 1} (-1)^i  \beta_i \, t^{i+e-1} \right ]. 
    \]
    
    Comparing the two expressions, we get
    \begin{align*}
        1 +  \sum_{i \ge 1} (-1)^i \beta_i \, t^{i+e-1} & = \sum_{i=0}^{m} f_{i-1} \, t^i \, (1-t)^{d-i} \\
        & = \sum_{i=0}^{m}  \left [ f_{i-1} \, t^i \, \sum_{j=0}^{d-i} (-1)^j \binom{d-i}{j} t^j \right ] \\[1ex]
        & = \sum_{i=0}^{m} t^i  \left [\sum_{j=0}^{i} (-1)^{i-j} f_{j-1}   \binom{d-j}{i-j} \right ]
    \end{align*}
    
    It follows that 
    \[
        \beta_i = \sum_{j = 0}^{i+e-1} (-1)^{e-1+j} f_{j-1}  \binom{d-j}{i+e-1-j},  
    \]
    as desired. In particular, we get $\beta_i = 0$ if $i \ge d-e+2$, which implies the depth estimate. 
\end{proof} 

\begin{rem}
    Observe that the above depth estimate is not true for arbitrary quotients of $R$. For example, the ideal $(x_1,\ldots,x_d)^e$ has an $e$-linear resolution, but $\depth R/(x_1,\ldots,x_d)^e = 0$. 
\end{rem} 

We are ready to determine the graded Betti numbers of the Stanley-Reisner ring associated to a linear space. 

\begin{thm}\label{thm:Betti lin space}
   Let $\Lambda$ be the incidence complex of a linear space $(P, L)$, and denote by $L_k$ the number of lines in $L$ with $k$ points. 
    Then, for each integer $i \ge 1$, 
    \begin{align*}
        \beta_{i} (K[\Lambda]) & = \beta_{i, i+2} (K[\Lambda])   \\
        & = \sum_{k \ge 1}  \binom{\# P-k}{i+2} L_k - \binom{\# P}{i+2} [ - 1 + \# L]  \\
        & \hspace*{0.6cm} + \binom{\# P-1}{i+1} \left [- \# P + \sum_{k \ge 2} k \, L_k  \right ] 
        - \binom{\# P-2}{i}  \left [- \binom{\# P}{2} + \sum_{k \ge 2} \binom{k}{2} L_k  \right ]. 
    \end{align*}
    
    In particular,  the depth of $K[\Lambda]$ is $2$, unless $\Lambda$ is a simplex, i.e., $\# L = 1$. 
\end{thm} 

\begin{proof} 
    Set $d = \# P$ and $\beta_i = \beta_{i} (K[\Lambda])$. By Proposition \ref{pro:icprop}, we know the $f$-vector of $\Lambda$. Thus, Corollary \ref{cor:3-lin} and  Proposition \ref{prop:betti e-linear} give, for any $i \ge 1$,
    \begin{align*}
        \beta_i & = \binom{d}{i+2} - d \binom{d-1}{i+1} + \binom{d}{2} \binom{d-2}{i} + \sum_{j=3}^{i+2} (-1)^j \left [\sum_{\ell \in L} \binom{\# \ell}{j} \right ] \binom{d-j}{i+2-j} \\
        & = \binom{d}{i+2} - d \binom{d-1}{i+1} + \binom{d}{2} \binom{d-2}{i} + \sum_{\ell \in L}  \sum_{j = 3}^{i+2} (-1)^{j} \binom{\# \ell}{j}  \binom{d -j}{i+2-j}.  
    \end{align*} 
    
    Using Lemma \ref{lem:sum-formula} below, we obtain
    \begin{align*}
        \beta_i & = \binom{d}{i+2} - d \binom{d-1}{i+1} + \binom{d}{2} \binom{d-2}{i} \\[.5ex]
        & \hspace*{0.6cm} + \sum_{k \ge 2} L_k \left [\binom{d-k}{i+2} - \binom{d}{i+2} + k \binom{d-1}{i+1} - \binom{k}{2} \binom{d-2}{i} \right ], 
    \end{align*} 
    which yields the desired formula for the Betti numbers. Furthermore, it follows
    \begin{equation*}
        \begin{split}
            \beta_{d-2}  & = \binom{d}{2} - d + 1 - \sum_{k \ge 1} L_k \left [\binom{k}{2} - k + 1   \right ] \\
            & =    \binom{d-1}{2} - \sum_{k \ge 2} L_k \binom{k-1}{2}.  \\
        \end{split}
    \end{equation*}
    
    Thus, the argument is complete once we have shown that 
    \[
        \binom{d-1}{2}  > \sum_{k \ge 2} L_k \binom{k-1}{2} = \sum_{\ell \in L} \binom{\# \ell -1}{2}, 
    \]
    unless $\Lambda$ is a simplex because then the projective dimension of $K[\Lambda]$ is $d-2$,  and so the depth of $K[\Lambda]$ is $2$, as desired. 

    Since any two points in $P$ lie on one and only one line, we get
    \[
         \binom{d}{2}=\sum_{\ell \in L} \binom{\# \ell}{2}. 
    \]
    
    If $\Lambda$ is not a simplex, then each point in $P$ is on at least two lines, which implies 
    \[
        d \le \sum_{\ell \in L} (\# \ell -1). 
    \]
    
    The last two displayed formulas along with Pascal's triangle equality easily give the needed estimate. 
\end{proof}

In the above argument we used the following identity. 

\begin{lem}\label{lem:sum-formula}
    Let $d, k$, and $b$ be nonnegative integers. Then 
    \[
        \sum_{j=0}^b (-1)^j \binom{k}{j} \binom{d-j}{b-j} = \binom{d-k}{b}. 
    \]    
\end{lem} 

\begin{proof} 
    This follows, for example, from \cite[Equation (6.18)]{Gould} by taking $b_j = (-1)^j, \ d = x+y, \ b=n$, and $k = x$. 
\end{proof} 

\begin{rem}\label{rem:number gen}
    For $i = 1$, one can simplify the formula in Theorem \ref{thm:Betti lin space}, using Lemma \ref{lem:sum-formula}. This gives  the expected number
    \[
        \beta_1(K[\Lambda]) = \binom{\# P}{3} - \sum_{k \ge 3} \binom{k}{3} L_k. 
    \]
\end{rem}     

In the case of a projective plane, the above formula for the Betti numbers becomes simpler.  

\begin{cor}\label{cor:Betti proj plane}
    Let $\Lambda$ be the incidence complex of a projective plane of order $q$. Then, for each integer $i \ge 1$, 
    \begin{align*}
        \beta_{i} (K[\Lambda]) & = \beta_{i, i+2} (K[\Lambda])   \\
        & = (q^2 + q +1) \binom{q^2}{i+2} + q^3 \binom{q^2 + q}{i+1} - q (q+1) \binom{q^2 + q}{i+2}. 
    \end{align*}     
    
    In particular, $\beta_{q^2 + q -1} (K[\Lambda]) = q^3$.
\end{cor} 

\begin{proof}
    Using that each line contains $q+1$ points and that  there are $q^2 + q + 1$ lines, this follows from Theorem \ref{thm:Betti lin space} and a straightforward computation. 
\end{proof} 

\begin{example}\label{exa:Betti plane of order 2} 
    Let $\Lambda$ be the incidence complex of a projective plane of order $q=2$.    Then its Stanley-Reisner ring has a minimal free resolution of the form 
    \[
        0 \to R(-7)^8 \to R(-6)^{42} \to R(-5)^{84} \to R(-4)^{77} \to R(-3)^{28} \to R \to K[\Lambda] \to 0. 
    \]
\end{example}


\section{The inverse system algebra associated to a linear space}\label{sec:is}

Let $(P, L)$ be a linear space, and set $d = \#P$.  We again assign to each point an indeterminate, $x_1, \ldots, x_d$, and let $R = K[x_1,\ldots,x_d]$, where $K$ is a field.  Later we will study how the characteristic of $K$ affects the properties
of the algebras that we will consider.

Each line $\ell \in L$ can be represented by a monomial $x_{i_1} \cdots x_{i_{\# \ell}}$ given by the product of the
indeterminates associated to the points on the line.  We thus obtain $e = \#L$ monomials $m_1, \ldots, m_e$.  Let
$I$ be the annihilator of $m_1, \ldots, m_e$ via \emph{inverse systems}, that is, $I$ consists of the polynomials that annihilate each of the monomials $m_1, \ldots, m_e$ under contraction. It is a monomial ideal.   Notice that in characteristic zero contraction is equivalent to the differentiation action (see \cite{IK}).  We will slightly abuse notation and view $I$ as an
ideal of $R$. Hence, for any positive integer $i$, a basis for $R/I$ is given by the monomials of degree $i$ that divide at least one of the $m_i$.  The set of these monomials generates the inverse system of $R/I$. 


\subsection{Minimal generators}

We first consider the minimal generators of $I$.

Since the monomials $m_1, \ldots, m_e$ are squarefree, they are annihilated by the squares of the indeterminates of which there are $d$.
Further, as any two points lie on a unique line, these are precisely the $d$ minimal generators of degree $2$.

Using this and noting that the inverse system consists of all  monomials corresponding to subsets of points on some line, 
we see that the generators of $I$ of degree $3$ are given by the squarefree monomials corresponding to three noncollinear points.

We now claim that these quadrics and cubics form a minimal generating set for $I$.  

\begin{prop}
  \label{prop:generators}
    The ideal $I$ is minimally generated in degrees $2$ and $3$; namely, a minimal generating set for $I$ consists of the $d$ squares of the 
    indeterminates, $x_1^2,\dots,x_d^2$, along with the squarefree monomials of degree $3$ that correspond to all possible triples of noncollinear points. 
\end{prop} 

\begin{proof} 
We have to show that $I$ does not have minimal generators whose degree is at least four. 
This comes from the fact that if 
four or more points of the space do not all lie on a line, then there is a subset of three that do not all lie on a line.  
Indeed, if $w,x,y,z$ do not all lie on a line but any three of them do, then in particular $\{w,x,y\}$ and $\{x,y,z\}$ are both 
sets of collinear points. Hence all of $w,x,y,z$ must lie on the line containing $x$ and $y$, since this is unique. 
\end{proof}

Proposition \ref{prop:generators} shows that  $I = I_{\Lambda} + (x_1^2, \ldots, x_d^2)$, where $I_{\Lambda}$ is the Stanley-Reisner ideal of the incidence complex $\Lambda$ 
of $(P, L)$ discussed in the previous section.  In particular, we have that the $h$-vector (or Hilbert function) of $R/I$ is $f(\Lambda)$, with $\dim_K [R/I]_1 = d$
and $\dim_K [R/I]_2 = \binom{d}{2}$.  

In the case when $(P, L)$ is a finite projective plane of order $q$, then $d = q^2 + q + 1$ and the $h$-vector of $R/I$ is
\begin{equation*} \label{hf}
    \left ( 1, d, d\cdot \binom{q+1}{2}, d \cdot \binom{q+1}{3} , \dots, d \cdot \binom{q+1}{q} , d \right )
\end{equation*}
(cf.\ \cite[page 38]{bjorner}).  Notice that this is a pure $O$-sequence (see for instance \cite{BMMNZ} for basic facts about pure $O$-sequences). In fact, the $h$-vector of $R/I$ is a pure $O$-sequence whenever the linear space $(P, L)$ is equipointed as then $R/I$ is level. 


\subsection{Minimal free resolution}

We now consider the minimal free resolution of $R/I$.  Let us start with some general remarks. 

Consider a simplicial complex $\Lambda$ on the vertex set $[d] = \{1,2,\ldots,d\}$. 
For a subset $F \subset [d]$, denote by $\fm^{\overline{F}}$ the ideal 
\[
\fm^{\overline{F}} = (x_i \: | \; i \in [d] \setminus F). 
\]
Then the Stanley-Reisner ideal of $\Lambda$ in $R = K[x_1,\ldots,x_d]$  can also be written as 
\[
I_{\Lambda} = \bigcap_{F \in \Lambda} \fm^{\overline{F}}, 
\]
where it is enough to take the intersection over the facets of $\Lambda$. 

For a subset $G$ of $[d]$, let $x_G$ be the monomial 
\[
x_G = \prod_{i \in G} x_i.
\]
Thus, we get for the colon ideal 
\begin{equation}
   \label{eq:colon formula}
I_{\Lambda} : x_G = \bigcap_{F \in \Lambda,\ G \subset F} \fm^{\overline{F}}.  
\end{equation}

Let now $\Lambda$ be the incidence complex to a linear space $(P, L)$, where $d = \# P$. 
 Then, for each $i$, the ideal 
\[
I_{\Lambda}  : x_i =  \bigcap_{\ell \in L,\ i \in \ell} \fm^{\overline{\ell}}
\]
 is the Stanley-Reisner ideal of a disjoint union of simplices. 

Now consider a subset $G$ of the vertex set with at least two elements. Then we get
\begin{equation} \label{colon}
I_{\Lambda}  : x_G =  \begin{cases}
 \fm^{\overline{\ell}} = (x_i \; | \; i \notin \ell) & \text{ if } G \subset \ell \\
R & \text{ otherwise}. 
 \end{cases}
\end{equation}
It follows in particular that $I_{\Lambda}  : x_G$ has a 1-linear resolution, unless it is the whole ring. 

Let  now 
\[
Q = (x_1^2,\ldots,x_d^2), 
\]
and set $I_{\Delta} = I_{\Lambda} + Q$. We want to determine the graded Betti numbers of $R/I_{\Delta}$. As preparation, we consider the link of $\Lambda$ with respect to any vertex. 

\begin{prop}
    \label{prop:Betti link} 
Let $\Lambda$ be the incidence complex of a linear space $(P, L)$, and denote by $L_k$ the number of lines in $L$ with $k$ points. 
Then, for each point $p$ of $P$ and each integer $i \ge 1$, 
\begin{align*}
\beta_{i} (K[\link_{\Lambda}(p)]) & = \beta_{i, i+1} (K[\link_{\Lambda}(p)])   \\
& =  \binom{\# P-1}{i+1} [ -1 + \# \{\ell \in L \; : \: p \in \ell\}] - \sum_{p \in \ell} \binom{\# P - \# \ell}{i+1}. 
\end{align*}
\end{prop} 

\begin{proof}
By Proposition \ref{pro:icprop}, $\link_{\Lambda}(p)$ is a disjoint union of simplices. The vertex of any such simplex is the set of points on a line containing $p$ other than $p$. Hence, its $f$-vector is
\[
f_{j-1} = \begin{cases} 
1 & \text{ if } j = 0\\
\sum_{p \in \ell} \binom{\# \ell - 1}{j} & \text{ if } j \ge 1. 
\end{cases}
\]

We know that the resolution of the Stanley-Reisner ideal of $\link_{\Lambda}(p)$ is 2-linear. Hence, Proposition \ref{prop:betti e-linear}  gives 
\begin{align*}
\beta_{i} (K[\link_{\Lambda}(p)]) & = \beta_{i, i+1} (K[\link_{\Lambda}(p)])   \\
& = \sum_{j=0}^{i+1} (-1)^{j+1} f_{j-1} \binom{\# P-1-j}{i+1-j} \\
& = - \binom{\# P-1}{i+1} + \sum_{p \in \ell} \sum_{j=1}^{i+1} (-1)^{j+1} \binom{\# \ell - 1}{j} \binom{\# P-1-j}{i+1-j}.
\end{align*}

Using Lemma \ref{lem:sum-formula}, we thus obtain
\begin{align*}
\beta_{i} (K[\link_{\Lambda}(p)]) & = - \binom{\# P-1}{i+1} + \sum_{p \in \ell} \left [ \binom{\# P - \# \ell}{i+1} - \binom{\# P-1}{i+1} \right ] \\
& =  \binom{\# P-1}{i+1} [ -1 + \# \{\ell \in L \; : \: p \in \ell\}] - \sum_{p \in \ell} \binom{\# P - \# \ell}{i+1}, 
\end{align*}
as desired. 
\end{proof} 

Again, the formula in the case of a projective plane becomes more explicit. 

\begin{cor}
    \label{cor:Betti link proj plane}
Let $\Lambda$ be the incidence complex of a projective plane of order $q$. Then, for each point $p \in P$ and each integer $i \ge 1$, 
\begin{equation*}
\beta_{i} (K[\link_{\Lambda}(p)]) = q \binom{q^2 + q}{i+1} - (q+1) \binom{q^2}{i+1}.
\end{equation*}
\end{cor}

We are ready for the main result of this section. 

\begin{thm}
     \label{thm:Betti w squares} 
Consider the ideal $I_{\Delta} = I_{\Lambda} + Q$, where $\Lambda$ is the incidence complex of a linear space $(P, L)$. Then 
$R/I_{\Delta}$ has graded Betti numbers
\[
\beta_{1, j} (R/I_{\Delta}) = \begin{cases}
\# P & \text{ if } j = 2 \\[.5ex]
\binom{\# P}{3} - \sum_{k \ge 3} \binom{k}{3} L_k & \text{ if } j = 3 \\[.5ex]
0 & \text{ otherwise},
\end{cases}
\]     
and, for each integer $i \ge 2$, 
\[
\beta_{i, j} (R/I_{\Delta}) = \begin{cases}
0 & \text{ if } j \le i+1 \\[.9ex]
 \binom{\# P}{i+2} - \# P \binom{\# P}{i+1} + \binom{\# P}{2} \binom{\# P - 2}{i}  \\[.5ex]
 \hspace*{.2cm} + \sum_{k \ge 2} L_k \left [ \binom{\# P-k}{i+2} - k \binom{\# P-k}{i} + \binom{k}{2} \binom{\# P-k}{i-2}  \right ] \\[.5ex]
 \hspace*{.2cm}  + \sum_{k \ge 2} L_k \left [ 
 - \binom{\# P}{i+2} + k \binom{\# P}{i+1} - \binom{k}{2} \binom{\# P-2}{i}  \right ] & \text{ if } j = i+2 \\[.9ex]
\sum_{k \ge j - i} \binom{k}{j-i} \binom{\# P - k}{2 i - j} L_k & \text{ if } j \ge i+3. 
\end{cases}
\]
\end{thm}

\begin{proof} 
Set $d = \# P$. 
By \cite{MPS}, the Koszul complex on the sequence $x_1^2,\ldots,x_d^2$ induces  an exact sequence of graded $R$-modules 
\[
0 \to F_{d} \to \cdots \to F_1 \to F_0 = R/I_{\Lambda} \to R/I_{\Delta} \to 0, 
\]
where (in our notation)
\[
F_i = \bigoplus_{G \subset [d],\ |G| = i} R/(I_{\Lambda} :x_G) (-2i). 
\]

Moreover, using mapping cones repeatedly gives a graded minimal free resolution of $R/I_{\Delta}$ by \cite[Theorem 2.1]{MPS}, and thus, for all integers $s$ and $t$,  
\begin{equation}
            \label{eq:iterated mapping cone}
\beta_{s, t} (R/I_{\Delta}) = \sum_{j = 0}^s \beta_{s-j, t} (F_j). 
\end{equation}

We know the graded Betti numbers of $F_0$ by Theorem \ref{thm:Betti lin space}. Thus, in order to apply the above formula,  we now determine the graded Betti numbers of $F_j$ for each $j \ge 1$. 

We begin by considering $F_1$. Note that the ideal $I_{\Lambda} : x_p$ is the extension ideal in $R$ of the Stanley-Reisner ideal of the link of $\Lambda$ with respect to the vertex $p$. Hence, $I_{\Lambda} : x_p$ has the same graded Betti numbers as $I_{\link_{\Lambda} p}$, which has a 2-linear resolution by Proposition \ref{prop:Betti link}. It follows that, for $i \ge 1$, 
\begin{align*}
\beta_i (F_1) & = \beta_{i, i+3} (F_1) \\
& = \sum_{p \in P} \beta_i (K[\link_{\Lambda} p]) \\
& = \sum_{p \in P}  \left [ [ -1 + \# \{\ell \in L \; : \: p \in \ell\}] \binom{d-1}{i+1} - \sum_{p \in \ell} \binom{d - \# \ell}{i+1}\right ] \\[.5ex]
& = \left [-d + \sum_{\ell \in L} \# \ell \right ] \binom{d-1}{i+1} - \sum_{\ell \in L} \# \ell \binom{d - \# \ell}{i+1} \\[.5ex]
& = \left [-d + \sum_{k \ge 2} k \, L_k \right ] \binom{d-1}{i+1} - \sum_{k \ge 2} k \, L_k \binom{d - k}{i+1}. 
\end{align*}

Summarizing, the following formula gives all the graded Betti numbers of $F_1$: 
\begin{equation}
    \label{eq:Betti F1} 
\beta_i (F_1) = \begin{cases} 
\beta_{0, 2} (F_1) = d & \text{ if } i = 0 \\
\beta_{i, i+3} (F_1) = {\displaystyle \left [-d + \sum_{k \ge 2} k \, L_k \right ] \binom{d-1}{i+1} - \sum_{k \ge 2} k \, L_k \binom{d - k}{i+1} } & \text{ if } i \ge 1. 
\end{cases}     
\end{equation}

Now we consider $F_j$, where $j \ge 2$.  Using Formula \eqref{colon}, we obtain 
\[
F_j = \bigoplus_{G \subset [d], \, |G| = j} (R/I_{\Lambda} : x_G) (-2j) \cong \bigoplus_{\ell \in L, \, G \subset \ell} R/(x_1,\ldots,x_{d - \# \ell}) (-2j). 
\]

Each direct summand has a linear resolution. Thus, for all nonnegative integers $i$ and $j \ge 2$,
\begin{equation}
     \label{eq:Betti Fj for bigger j}
\begin{split}     
\beta_i (F_j)  = \beta_{i, i + 2j} (F_j)  & = \sum_{\ell \in L} \binom{\# \ell}{j} \binom{d - \# \ell}{i} \\
& = \sum_{k \ge j} \binom{k}{j} \binom{d-k}{i} L_k. 
\end{split}
\end{equation}

We are now ready to apply Formula \eqref{eq:iterated mapping cone}. For $s = 1$, we have
\begin{align*}
\beta_{1, t} (R/I_{\Delta}) & = \beta_{0, t} (F_1) + \beta_{1, t-1} (F_0) \\
& = \begin{cases}
\beta_{0, 2} (F_1) = d & \text{ if } t = 2 \\
\beta_{1, 3} (F_0)  & \text{ if } t = 3 \\
0 &  \text{ otherwise}.
\end{cases}
\end{align*}

Thus, Remark \ref{rem:number gen} gives the desired first graded Betti numbers of $R/I_{\Delta}$. 

Let now $s \ge 2$. Then the above formulas and Theorem \ref{thm:Betti lin space} 
 provide, for each integer $t$, 
\begin{align*}
\beta_{s, t} (R/I_{\Delta}) & = \begin{cases}
0 & \text{ it } t \le s+1 \\
\beta_{s, s+2} (F_0) + \beta_{s-1, s+2} (F_1) + \beta_{s-2, s+2} (F_2) & \text{ if } t = s+2 \\[.5ex]
\beta_{2s - t, t} (F_{t-s}) = \sum_{k \ge t - s} \binom{k}{t-s} \binom{d-k}{2s - t} L_k  & \text{ if } t \ge s+3.
\end{cases}
\end{align*}

Therefore, it only remains to determine $\beta_{s, s+2} (R/I_{\Delta})$. We have, for $s \ge 2$, 
\begin{align*}
\beta_{s, s+2} (R/I_{\Delta}) & = \sum_{k \ge 1}  \binom{d-k}{s+2} L_k - \binom{d}{s+2} [ - 1 + \# L]  \\
& \hspace*{0.6cm} + \binom{d-1}{s+1} \left [- d + \sum_{k \ge 2} k \, L_k  \right ] 
- \binom{d-2}{s}  \left [- \binom{d}{2} + \sum_{k \ge 2} \binom{k}{2} L_k  \right ] \\
& \hspace*{0.6cm} + \left [-d + \sum_{k \ge 2} k \, L_k \right ] \binom{d-1}{s} - \sum_{k \ge 2} k \, L_k \binom{d - k}{s} \\
& \hspace*{0.6cm} + \sum_{k \ge 2} \binom{k}{2} \binom{d-k}{s-2} L_k\\[.5ex]
& = \binom{d}{s+2} - d \binom{d-1}{s+1} + \binom{d}{2} \binom{d - 2}{s} - d \binom{d-1}{s} \\[.5ex]
& \hspace*{0.6cm} + \sum_{k \ge 2} L_k \left [ \binom{d-k}{s+2} - k \binom{d-k}{s} + \binom{k}{2} \binom{d-k}{s-2}  \right ] \\[.5ex]
&  \hspace*{0.6cm} + \sum_{k \ge 2} L_k \left [ - \binom{d}{s+2} + k \binom{d-1}{s+1} - \binom{k}{2} \binom{d-2}{s} + k \binom{d-1}{s}
 \right ], 
 \end{align*}
which implies the desired formula. 
\end{proof}

\begin{rem}
Some of the Betti numbers of $R/I_{\Delta}$ have an easy combinatorial interpretation. For instance, if $i \ge 2$ and $j \ge 3$, then the formula for $\beta_{i, j} (R/I_{\Delta})$ in Theorem \ref{thm:Betti w squares} can be rewritten as 
\[
\beta_{i, i+j} (R/I_{\Delta}) = \sum_{\ell \in L} \binom{\# \ell}{j} \binom{\# P - \# \ell}{i - j}. 
\]

Thus, $\beta_{i, i+j} (R/I_{\Delta})$ equals the number of choices of $i$ points in $P$ such that $j$ points are on one line in $L$ and $i - j$ points are off the chosen line. 
\end{rem}


We highlight again the case of a projective plane. 

\begin{cor}
    \label{cor:Betti proj plane w squares} 
For a projective plane of order $q$, the graded Betti numbers of     $R/I_{\Delta}$ are
\[
\beta_{1, j} (R/I_{\Delta}) = \begin{cases}
q^2 + q + 1 & \text{ if } j = 2 \\[.5ex]
\binom{q^2 + q + 1}{3} -  \binom{q+1}{3} (q^2 + q + 1) & \text{ if } j = 3 \\[.5ex]
0 & \text{ otherwise},
\end{cases}
\]     
and, for each integer $i \ge 2$, 
\[
\beta_{i, j} (R/I_{\Delta}) = \begin{cases}
0 & \text{ if } j \le i+1 \\[.9ex]
  - q (q + 1) \binom{q^2 + q + 1 }{i+2}  \\[.5ex]
 \hspace*{.2cm} + (q^2 + q + 1) \left [ \binom{q^2}{i+2} + q \binom{q^2 + q + 1}{i+1} - (q+1) \binom{q^2 }{i} + \binom{q+1}{2} \binom{q^2}{i-2}  \right ] & \text{ if } j = i+2  \\[.9ex]
(q^2 + q + 1)  \binom{q+1}{j - i} \binom{q^2}{2i- j} & \text{ if } j \ge i+3. 
\end{cases}
\]    
In particular, $\beta_{i, j} (R/I_{\Delta})$ is not zero if and only if 
\begin{enumerate}
\item $i = j = 0$ or $i+1 = j = 2$; 
\item $1 \le i \le q^2 + q$ and $j = i+2$; \quad or 
\item $i \ge 2$ and $\max \{i+3, 2i - q^2\} \le j \le \min \{2i, i+q+1\}$. 
\end{enumerate}
\end{cor}

\begin{proof}
This  follows from Theorem \ref{thm:Betti w squares} by a straightforward computation. 
\end{proof}

Of course, analogous results to those in Corollaries~\ref{cor:Betti link proj plane} and~\ref{cor:Betti proj plane w squares} 
could also be stated, for instance, for finite affine planes. 

\begin{example}
To illustrate the above computation,  consider  a projective plane of order $q=2$. Then we have, using the above notation, 
\[
F_2 \cong  (R/(x_1,\dots,x_4))^{21} (-4), \ \ F_3 \cong  (R/(x_1,\dots,x_4))^{7}  (-6), \ \ F_j = 0 \; \text{ if } j \ge 4. 
\]

Hence, we obtain the following diagram: 
\[
\begin{array}{ccccccccccccccccc}

0 & \rightarrow & F_3 & \rightarrow & F_2 & \rightarrow & F_1 & \rightarrow & R/I_\Lambda & \rightarrow  & R/I_\Delta & \rightarrow & 0 \\
&& \uparrow && \uparrow &&\uparrow &  & \uparrow  \\
&& R(-6)^7 && R(-4)^{21} && R(-2)^7 && R \\
&& \uparrow && \uparrow &&\uparrow &  & \uparrow  \\
&& R(-7)^{28} && R(-5)^{84} && R(-4)^{84} && R(-3)^{28} \\
&& \uparrow && \uparrow &&\uparrow &  & \uparrow  \\
&& R(-8)^{42} && R(-6)^{126} && R(-5)^{196} && R(-4)^{77} \\
&& \uparrow && \uparrow &&\uparrow &  & \uparrow  \\
&& R(-9)^{28} && R(-7)^{84} && R(-6)^{189} && R(-5)^{84} \\
&& \uparrow && \uparrow &&\uparrow &  & \uparrow  \\
&& R(-10)^{7} && R(-8)^{21} && R(-7)^{84} && R(-6)^{42} \\
&& \uparrow && \uparrow &&\uparrow &  & \uparrow  \\
&& 0 && 0 && R(-8)^{14} && R(-7)^{8} \\
&&  &&  &&\uparrow &  & \uparrow  \\
&&&&&& 0 && 0

\end{array}
\]


It gives the following Betti table for $R/I_\Delta$:
\begin{verbatim}
        0    1     2     3     4     5    6   7
-------------------------------------------------------
 0:     1    -     -     -     -     -    -   -
 1:     -    7     -     -     -     -    -   -
 2:     -   28   182   364   357   176   35   -
 3:     -    -     -     7    28    42   28   7
-------------------------------------------------------
Tot:    1   25   182   371   385   218   63   7
 \end{verbatim}
 
 Notice that it is in accordance with Corollary \ref{cor:Betti proj plane w squares}. 
\end{example}

\begin{rem}
    Let $I$ be the inverse system of a finite projective plane.
    The last line of the Betti diagram of $R/I$ is
    \begin{center}
        \begin{tabular}{c|ccccccccccccccccccccccc}
        & $0$ & $1$ & $2$ & \dots & $q$ &  $q+1$ & $q+2$ & $q+3$ & \dots & $q^2+q-1$ & $q^2+q$ & $q^2+q+1$   \\ \hline \\
        $q+1$ & - & - & - & \dots & $0$ & $d$ & $d \cdot \binom{q^2}{q^2-1}$ & $d \cdot \binom{q^2}{q^2-2}$ & \dots & $d \cdot \binom{q^2}{2}$ & $d \cdot \binom{q^2}{1}$ & $d$
        \end{tabular}
    \end{center}
    While this follows as a consequence of the above, it also follows by looking at the minimal free resolution directly.

    We consider a minimal free resolution for the canonical module, $M$, which we will view as the inverse system module 
    generated by the monomials $m_i$. We know that $M$ is generated in the initial degree, and has $d$ minimal generators.  
    Furthermore,  each generator $m_i$ corresponding to a line $\ell_i$ is annihilated by the indeterminates  corresponding 
    to points that are not on $\ell_i$.  There are $q^2+q+1 - (q+1) = q^2$ such indeterminates.

    In the minimal free resolution of $M$, then, the first free module is $R^d$ (ignoring the twist). The linear syzygies 
    are exactly given by the annihilation just described. For each monomial $m_i$, $1 \leq i \leq d$, let 
    $x_{i,1},\dots,x_{i,q^2}$ be the indeterminates that annihilate $m_i$.
    Then the presentation matrix has the form
    \[
        \left [
            \begin{array}{ccccccccccccccccccccccccc}
                x_{1,1} & \dots & x_{1,q^2} & 0 & \dots & 0 & 0 & \dots & 0 & \dots & 0 & \dots & 0 & \dots    \\
                0 & \dots & 0 & x_{2,1} & \dots & x_{2,q^2} & 0 &  \\
                0 & \dots & 0  & 0 & \dots & 0 & x_{3,1} & \dots & x_{3,q^2} & \dots & 0 & \dots & 0 & \dots    \\
                \vdots  \\
                0 & \dots & 0 & 0 & \dots & 0 & 0 & \dots & 0 & \dots & x_{d,1} & \dots & x_{d,q^2} & \dots 
            \end{array}
        \right ]
    \]
    (where there may be additional syzygies of higher degree).  The syzygies of the columns of this matrix are precisely the 
    Koszul syzygies of a complete intersection of $q^2$ linear forms, and there are $d$ such complete intersections.
\end{rem}


\section{The Lefschetz Properties}\label{sec:lefschetz}

The purpose of this section is to determine in which characteristics of the base field $K$, which we will assume throughout to be infinite,  the
algebra $A=R/I_{\Delta}$ associated to a linear space, as defined in Section \ref{sec:is}, has the WLP or the SLP.  We will do this by applying a variety of tools coming from combinatorial commutative algebra to the structure of linear spaces.  The resulting characterizations will, in some cases, turn out to be extremely appealing. Namely, our main result, which assumes $A$ is  associated to a finite projective plane of order $q$, will 
show that $A$ has the WLP if and only if $\Char{K} = 0$ or $\Char{K} > \left\lceil \frac{q+1}{2} \right\rceil$ (as usual, $\lceil x \rceil$ denotes
the least integer $\geq x$), while $A$ has the SLP if and only if $\Char{K} = 0$ or $\Char{K} > q+1$ (see Theorem~\ref{main} (i) and (iv)).

\begin{rem}
    By Theorem~\ref{thm:Betti lin space}, for any linear space the algebra $R/I_{\Lambda}$ has positive depth. Hence, it always has both Lefschetz Properties. Indeed (if the base field is infinite) the fact that the depth of $R/I_{\Lambda}$ is positive guarantees the existence of a \emph{linear} nonzero divisor in $R/I_{\Lambda}$, which immediately gives us that all maps defining the Lefschetz Properties are injective.
\end{rem}

We first  present some preliminary results and helpful facts.

\begin{lemma}[\cite{MMN-mon}]\label{mmn}
    Let $A$ be any (standard graded) artinian algebra, and let $\ell$ be a general linear form.  
    Consider the maps $\times \ell : A_i \rightarrow A_{i+1}$ defined by multiplication by $\ell$, for $i \geq 0$.  Then:
    \begin{enumerate}
        \item If $\times \ell$ is surjective from some degree $i$ to degree $i+1$, then $\times \ell$ is surjective in all subsequent degrees.
        \item If $A$ is level and $\times \ell$ is injective from some degree $i$ to degree $i+1$, then $\times \ell$ is injective in all previous degrees.
        \item In particular, if $A$ is level and $\dim_K A_{i} = \dim_K A_{i+1}$, then $A$ has the WLP if and only if $\times \ell$ is injective (i.e., bijective) from  degree $i$ to degree $i+1$.
    \end{enumerate}
\end{lemma}

\begin{rem} 
    If the algebra $A$ is \emph{monomial}, it was shown in  \cite[Proposition 2.2]{MMN-mon} that, without loss of generality, 
    one may simply assume, in studying the existence of the WLP or of the SLP, that the linear form $\ell$ is the sum of the
    indeterminates.  Thus, from now on, when we refer to $\ell$, it will be to the specific linear form $\ell=x_1+\dots +x_d$.
\end{rem}

We also recall that, for the ideal of the squares of the variables, the presence of the WLP and SLP  admits a nice classification.

\begin{lem}\label{lem:squares}
    If $A = K[x_1, \ldots, x_d] / (x_1^2, \ldots, x_d^2)$, then
    \begin{enumerate}
        \item $A$ has the WLP if and only if $\Char{K} = 0$ or $\Char{K} > \left \lceil \frac{d}{2} \right \rceil$.
        \item $A$ has the SLP if and only if $\Char{K} = 0$ or $\Char{K} > d$.
    \end{enumerate}
\end{lem}
\begin{proof}
    The characteristic zero portion of each part follows from Stanley's theorem in \cite{St2}.

    The positive characteristic portion of (i) follows from Theorem 6.4 of \cite{KV}, and the positive characteristic portion of
    (ii) follows from one of Theorem 5.5, Proposition 6.7, and Corollary 6.5 of \cite{cook} (see also Theorem 7.2 therein) depending on $d$.
\end{proof}

Using the same notation as in the previous section, let us now consider the artinian monomial algebra $A = R/I_{\Delta} = K[x_1,\dots,x_d]/I_{\Delta} $ 
corresponding to a linear space $(P, L)$ with $d = \#P$.  Notice that $A$ is level if and only if  $(P, L)$ is equipointed.

Hence Lemma~\ref{lem:squares} classifies the case when  $\#L = 1$ (which includes the case $d=2$). In particular, if $d=2$  then $A$ always has the WLP, and has the SLP if and only if $\Char{K} \neq 2$.
We thus assume for the remainder of the section that $d > 2$.

We first note that the second and last map for the WLP are simple.

\begin{lemma}\label{ciao}
  Suppose that the largest number of points on any line is $m$.
    \begin{enumerate}
        \item The map $\times \ell :A_1 \rightarrow A_2$ is injective if and only if $\Char{K} \neq 2$.
        \item The map $\times \ell :A_{m-1} \rightarrow A_{m}$ is surjective in all characteristics, for $m \geq 3$.
    \end{enumerate}
\end{lemma}

\begin{proof}
    (i) Let us choose as a basis for $A_1$ the (homomorphic images of the) indeterminates of $R$, and as a basis for $A_2$ the 
    squarefree monomials of degree 2. Thus, the matrix associated to $\times \ell$ with respect to these bases is a 
    $\binom{d}{2} \times d$ incidence matrix, whose entries are $0$ or $1$ depending on whether the indeterminate divides the 
    degree $2$ monomial.  In particular, each row contains exactly two $1$'s, and each column contains exactly $(d-1)$ $1$'s.
    
    It is easy to see that injectivity fails if and only if the columns of $\times \ell$ are linearly dependent, if and only if 
    the first column is a linear combination of the others.  A brief thought gives that this is equivalent to the first column
    being equal to the sum of the other columns, in order to account for the $1$'s.  But then, in any fixed row corresponding 
    to a $0$ in the first column, there are precisely two $1$'s, neither of which is in the first column.  The sum of these $1$'s
    is therefore 0, which can occur if and only if the characteristic of $K$ is $2$, as desired.

    (ii) For any nonzero element $f = x_{i_1} \cdots x_{i_{m}}$ of $A_m$, it easy to see that $f = \ell \cdot x_{i_1} \cdots x_{i_{m-1}}$, since $x_i^2 = 0$ for all $i$ and there exists at most one line passing through any $m-1$ points.
    This proves the result.
\end{proof}
   
\begin{cor}
    If every line contains exactly two points (i.e. $A$ has socle degree $2$) then $A$ has the WLP (SLP) if and only if $\Char{K} \neq 2$.
\end{cor}

Thus we may assume for the remainder of the section that the socle degree $m$ is at least $3$.


\subsection{Equipointed linear spaces}

Suppose $A$ is level, i.e., suppose every line in $L$ has exactly $m$ points.  In this case, each of the $d$ vertices is part of $\frac{d-1}{m-1}$ lines,
each of which is made of $m$ vertices, i.e., $\#L = \frac{d(d-1)}{m(m-1)}$.  Thus the $h$-vector of $A$ is
\[
    \left(1, d, \frac{d(d-1)}{m(m-1)} \binom{m}{2}, \ldots, \frac{d(d-1)}{m(m-1)} \binom{m}{m} \right),
\]
which is unimodal with peak in degree $\frac{m}{2}$ if $m$ is even and peaks in degrees $\frac{m-1}{2}$ and $\frac{m+1}{2}$ if $m$ is odd.

\begin{rem}\label{deg2}
Let $C$ be an artinian complete intersection generated by the squares of $m$ indeterminates. 
Let $B$ be the $R$-submodule of $A$, defined by  $B = \oplus_{j=2}^m [A]_j$,  and similarly let $D = \oplus_{j=2}^m [C]_j$. Then $B$ decomposes, as an $R$-module, as a direct sum of $\# L$ copies of $D$.
    This is easily shown via inverse systems using the monomial 
    generators corresponding to the $\#L$ lines, with the direct sum being a consequence of the fact that any two points lie on a unique line. 
    As a result of this, the injectivity and/or surjectivity of most maps $\times \ell^j$ on $A$  follows from the corresponding result 
    for a monomial complete intersection of quadrics. More precisely, we have that \emph{the map $\times \ell^j$ between  $A_i$ and $A_{i+j}$ 
    has maximal rank if and only if the corresponding map between $C_i$ and $C_{i+j}$ has maximal rank, for any given $i\ge 2$ and $j\ge 0$}.
\end{rem}

We are now ready for the main theorem of this section.


\begin{thm}\label{main}
    Suppose $A$ is level with socle degree $m \geq 3$, coming from a linear space with $\#L \geq 2$.
    \begin{enumerate}
        \item $A$ has the WLP if and only if $\Char{K} = 0$ or $\Char{K} > \left\lceil \frac{m}{2} \right\rceil$.
        \item $A$ does not have the SLP if $2 \leq \Char{K} \leq m$.
        \item $A$ has the SLP if $\Char{K} = 0$ or $\Char{K} > \max\{m, \frac{d-1}{m-1} \} = \frac{d-1}{m-1} $.
        \item If $A$ is associated to a finite projective plane, then 
            $A$ has the SLP if and only if $\Char{K} = 0$ or $\Char{K} > m$.
    \end{enumerate}
\end{thm}

\begin{proof}
    (i) If $m=3$, notice that, since the algebra $A$ has socle degree $m=3$, by Lemma \ref{ciao} we easily have that $A$ 
    has the WLP if and only if $\Char{K} \neq 2$, which proves the statement.

    Let $m=4$. Now the socle degree of $A$ is $m=4$ and, by Lemma \ref{ciao}, injectivity fails for $A$ from degree 
    $1$ to degree $2$ if and only if $\Char{K} = 2$. Since the $h$-vector of $A$ reaches its unique peak in degree $2$, 
    by Lemma \ref{mmn} it  suffices to determine the characteristics in which surjectivity fails for $A$ from degree $2$ 
    to degree $3$. But by Remark \ref{deg2}, this is equivalent to determining the same result for the algebra 
    $C=K[x_1,\dots,x_4]/(x_1^2,\dots,x_4^2)$. This is essentially done in \cite[Theorem 5.1]{KV}, where it is shown that $C$ 
    has the WLP in any characteristic larger than $2$ (in the notation of \cite[Formula (4.3)]{KV},  set $k=q=r=1$ and $d=2$ to 
    immediately obtain  the result). Thus, since in characteristic zero $C$ is well known to have the SLP and therefore the WLP 
    (see Stanley's paper \cite{St2}), we conclude that $A$ has the WLP if and only if $\Char{K} \neq 2$, as desired.

    Suppose now that $m\ge 5$ is odd. Notice that the $h$-vector of $A$ is unimodal with two peaks in its middle degrees, 
    $\frac{m-1}{2}$ and $\frac{m+1}{2}$. Therefore, by Lemmas \ref{mmn} and \ref{ciao}, it suffices to show in which characteristics  
    the map $\times \ell$ is injective (i.e., bijective) from degree  $\frac{m-1}{2}$ to degree $\frac{m+1}{2}$.
    By Lemma~\ref{lem:squares}, the algebra $C = K[x_1,\dots,x_{m}]/(x_1^2,\dots,x_{m}^2)$ has the WLP if and only if 
    $\Char{K} = 0$ or $\Char{K} > \left \lceil \frac{m}{2} \right \rceil$.  Thus the result follows from Remark~\ref{deg2}.

    Finally, let $m \ge 6$ be even. Hence, the $h$-vector of $A$ is unimodal with a unique peak in degree $\frac{m}{2}$. Reasoning in a 
    similar fashion to the previous cases, it is  enough to determine in which characteristics  the map $\times \ell$ is injective from 
    degree $\frac{m-2}{2}$ to degree $\frac{m}{2}$, and in which characteristics it is surjective from degree 
    $\frac{m}{2}$ to degree $\frac{m+2}{2}$.  But this can again be achieved by invoking Remark \ref{deg2}, then Lemma~\ref{lem:squares}.

    (ii) First, consider the maps $\ell^j: A_0 \rightarrow A_j$.  Proving these maps are all injective is tantamount to showing that the form
    $\ell^j$ itself is nonzero.  If we fix any squarefree monomial $M$ of degree $j$ corresponding to $j$ collinear points (which exist since $j \leq m$),
    then $M \neq 0$, and by the multinomial theorem, the coefficient of $M$ in $\ell^j$ is $j!$, which is obviously nonzero in $\ZZ$, and 
    nonzero modulo any prime $p > m$.  Thus, in particular, $\ell^{m}$ is zero modulo any prime $p$ if $p \leq m$.  Thus $A$ fails to have the SLP
    if $2 \leq \Char{K} \leq m$.  

    (iii) We first notice that $\frac{d-1}{m-1}   \geq m$. Indeed, it was noted above that each point, $Q$, of our linear space lies on $\frac{d-1}{m-1}$ lines. Among these are the $m$ lines joining $Q$ to a point on a line not containing $Q$, but there may be more lines through $Q$ that do not meet this line. In fact, apart from trivial cases (where there are too few points), the linear space is a projective plane if and only if $\frac{d-1}{m-1}   = m$.
    
    We now want to show that $A$ has the SLP in most of the other  characteristics. First of all, recall that, because of Remark \ref{deg2}, 
    the map $\times \ell^j$ between $A_i$ and $A_{i+j}$  has maximal rank if and only if the corresponding map between  graded components of 
    $C=k[x_1,\dots,x_{m}]/(x_1^2,\dots,x_{m}^2)$ does, for any $i\ge 2$ and any $j\ge 0$.  Again, we invoke Lemma~\ref{lem:squares}, 
    which precisely guarantees that the above maps have maximal rank for the algebra $C$ in characteristic zero and in all characteristics greater 
    than $m$. 
    
    Thus, assuming from now on that $\Char{K}=0$ or $\Char{K}>m$, it remains to prove that, when $i=1$, all maps 
    $\times \ell^j$ between $A_1$ and $A_{1+j}$ have maximal rank (i.e., they are all injective).  We first consider
    $j = m-1$.  In this case, the matrix $B$ for the multiplication map $\times \ell^{m-1}: A_1 \rightarrow A_{m}$ is 
    $(m-1)!$ times the incidence matrix $M$. Indeed, a basis of $A_1$ is given by the residue classes of the $d$ variables corresponding to the points in the linear space. The products of $m$ variables corresponding to points on a line form a basis of $A_m$. Hence, all non-squarefree summands in the expansion of $\ell^{m-1}$ are mapped to zero. The coefficient of the squarefree monomials appearing  in $\ell^{m-1}$ is $(m-1)!$. Using     
    $(m-1)! M = B$, it follows that if $\hbox{char } K > m-1$ then $M$ has maximal rank if and only if $B$ has maximal rank.  Furthermore, if $M^T M$ has maximal rank then $M$ has maximal rank. (Note that this last fact holds only when $\#L \geq \#P$, which in turn is true because of the assumption that $\#L \geq 2$.)
    
    The matrix $M^{\rm T}M$ is square with rows and columns index by the points of $(P, L)$.  In particular, the $(i,j)$-th entry
    is the number of lines containing both the $i$-th and the $j$-th points; i.e., it equals $\frac{d-1}{m-1}$ if $i = j$, and $1$ otherwise.
    
Thus $M^{\rm T}M$ has determinant
    \begin{equation*}
        \begin{split}
            \det(M^{\rm T}M)
                &= 
                \left| 
                    \begin{array}{cccc}
                        \frac{d-1}{m-1} & 1 & \cdots & 1 \\
                        1 & \frac{d-1}{m-1} & \cdots & 1 \\
                        \vdots &  & \ddots & \vdots \\
                        1 & 1 & \cdots & \frac{d-1}{m-1} \\
                    \end{array}
                \right| \\
                 &= 
                 \left| 
                     \begin{array}{ccccc}
                         \frac{d-1}{m-1} & 1 & 1 &  \cdots & 1 \\
                         1 - \frac{d-1}{m-1} & \frac{d-1}{m-1} - 1 & 0 & \cdots & 0 \\
                         \vdots &  & \ddots & \vdots \\
                         1 - \frac{d-1}{m-1} & 0 & 0 & \cdots & \frac{d-1}{m-1} - 1 \\
                     \end{array}
                 \right| \\
                 &= 
                 \left| 
                     \begin{array}{ccccc}
                         \frac{d-1}{m-1} + d - 1 & 1 & 1 &  \cdots & 1 \\
                         0 & \frac{d-1}{m-1} - 1 & 0 & \cdots & 0 \\
                         \vdots &  & \ddots & \vdots \\
                         0 & 0 & 0 & \cdots & \frac{d-1}{m-1} - 1 \\
                     \end{array}
                  \right| \\
                &= \left( \frac{d-1}{m-1} + d - 1 \right)\left( \frac{d-1}{m-1} - 1 \right)^{d-1} \\
                &= m \cdot \frac{d-1}{m-1} \cdot \left( \frac{d-m}{m-1} \right)^{d-1}.
        \end{split}
    \end{equation*}
    
    Hence $M^{\rm T}M$ has a nonzero determinant in $\ZZ$ and modulo a prime $p$ if $p > \max\{m, \frac{d-1}{m-1} \} = \frac{d-1}{m-1}$, 
    and so, in particular, $\times \ell^{m-1}: A_1 \rightarrow A_m$ is injective if $\Char{K} = 0$ or 
    $\Char{K} > \frac{d-1}{m-1}$.

    Finally, assume $\times \ell^j: A_1 \to A_{j+1}$ is not injective for some $1 \leq j \leq m-2$.  Thus $\ell^j \cdot x = 0$ for some nonzero linear form $x$.
    But this implies that $\ell^{m-1} \cdot x= \ell^{m-1-j} \cdot \ell^j \cdot x = 0$, contradicting the injectivity of $\times \ell^{m-1}$ 
    proven above in characteristic zero or characteristic $> \frac{d-1}{m-1}$. This completes the proof of this case.

    (iv) Suppose $A$ is associated to a finite projective plane of order $q = m - 1$.  In this case, the matrix $M$ is square
    and so its determinant is the square root of the determinant of $M^{\rm T}M$, i.e., $(q+1)q^{(q^2+q)/2}$.  Hence
    \[
        \det(B) = (q!)^{q^2+q+1} \cdot (q+1) \cdot q^{(q^2 + q)/2}.
    \]
    
    Thus, $\times \ell^q: A_1 \rightarrow A_{q+1}$ is an isomorphism if and only if $\Char{K} = 0$ or $\Char{K} > q+1 = m$, as desired.
\end{proof}

We note that the determinant of the matrix $M$ in the case of finite projective planes was given in  \cite[Lemma 29]{Kaa}.

\begin{rem}
We do not know if part (iii) of the above theorem is true even if we only assume $\hbox{char } K > m$. An interesting test case is the affine plane of order 4, where we have $d = 16$, $m = 4$, and $\frac{d-1}{m-1} = 5$. Thus if we take $\hbox{char } K = 5$, we see that $\det(M^T M) = 0$ so we might hope that in characteristic 5 $A$ does not have the SLP (thus providing a counterexample to the question just asked). But in fact we have verified experimentally that $A$ {\em does} have the SLP. This is because we only know that $M^T M$ having maximal rank implies that $M$ does, but not necessarily the converse. The converse can fail not only in characteristic $p$ but in fact also over the complex numbers. Interestingly, it is true over the real numbers, since for a real vector $x$, the number
$x^T M^T M x$ is the square of the length of $Mx$.
\end{rem}


\subsection{Nonequipointed linear spaces}

Now we consider the case when the linear space has lines of different sizes.

Our main result for nonequipointed linear spaces is that the WLP only holds for ``almost equipointed'' linear spaces.

\begin{thm}\label{almost-equipointed}
    Let $(P, L)$ be a linear space with $d = \#P$.  Set $n$ and $m$ to be the minimum and maximum, respectively, size of a line in $L$.
    Suppose $2 \leq n < m$.  Then the ring $A$ has the WLP if and only if (a) $m \leq 5$ or $n + 1 = m$, and (b) $\hbox{char } K > \left\lceil \frac{m}{2} \right\rceil$ or $\hbox{char } K = 0$.
\end{thm}
\begin{proof}
    We must consider the maps $\times\ell: A_{i} \rightarrow A_{i+1}$.  Clearly, if $i = 0$, then the map always has maximal rank. 
    By  Lemma~\ref{ciao}, if $i = 1$, then the map has maximal rank if and only if $\Char K \neq 2$, and if $i = m-1$, then the map always has maximal rank.
    Thus we must consider the maps when $2 \leq i < m-1$.

    In each of these cases, the matrix $M_i$ for the map $\times\ell: A_{i} \rightarrow A_{i+1}$ has columns indexed by $i$-tuples of
    collinear points and rows indexed by $(i+1)$-tuples of collinear points.  Since $i \geq 2$, the matrix $M_i$ is a block diagonal matrix
    with blocks indexed by the lines in $L$. Let $B_c$ be the block associated to the line $c$ in $L$.  The matrix $B_c$ is 
    $\binom{\#c}{i+1} \times \binom{\#c}{i}$.  It is also the matrix for the map $\times\ell: C_i \rightarrow C_{i+1}$, where 
    $C = K[x_1, \ldots, x_{\#c}] / (x_1^2, \ldots, x_{\#c}^2)$. Invoking Lemma~\ref{lem:squares}, we see that such maps all have
    maximal rank if $\hbox{char } K > \left \lceil \frac{\#c}{2} \right \rceil$.

    Clearly, all of the blocks $B_c$ for all matrices $M_i$ will have maximal rank if and only if $\Char{K} = 0$ or 
    $\Char{K} > \left \lceil \frac{m}{2} \right\rceil$.  However, for $M_i$ we also need to have all of the blocks $B_c$ to be
    either simultaneously injective or simultaneously surjective.  It is an easy exercise in comparing the peaks of binomial coefficients
    to see that this can only occur if $m \leq 5$ or $n + 1 = m$.
\end{proof}

Based on experimental evidence,  we offer a conjecture as to the presence of the SLP.  

\begin{conj}
    Let $(P, L)$ be a linear space with $d = \#P$.  Set $n$ and $m$ to be the minimum and maximum, respectively, size of a line in $L$.
    Suppose $2 \leq n < m$.  Then the ring $A$ has the SLP if and only if $\Char{K} = 0$ or $\Char{K} > m$ and one of the following conditions holds:
    \begin{enumerate}
        \item $3 \leq m \leq 4$,
        \item $m = 5$ and every point is on a line of size at least $4$,
        \item $m \geq 6$ is even and $n+1 = m$.  
    \end{enumerate}
\end{conj}

We close with an interesting example.

\begin{example}\label{n8}
    Let $(P, L)$ be the linear space  consisting of 2 disjoint lines with 8 points each, and all the other necessary lines of size $2$ between them.
    Then the associated artinian algebra $A = R/I_{\Delta}$ has $h$-vector 
    \[
        (1, 16, 120, 112, 140, 112, 56, 16, 2),
    \]
    which is nonunimodal, forcing $A$ to fail to have the WLP, which is consistent with Theorem~\ref{almost-equipointed}. More generally, and again consistently with Theorem~\ref{almost-equipointed}, an ``$(m,n)$-bipartite" linear space gives an algebra with $mn$ socle elements in degree 2 (coming from the lines with two points), so when $m \geq 6$ the multiplication from degree 2 to degree 3 fails to have maximal rank.
    \end{example}

\begin{rem}
The previous example confirms that the nature of arbitrary linear spaces can be very different from the subclass of projective planes, but also that they can produce some interesting $O$-sequences (or equivalently, some interesting $f$-vectors of simplicial complexes). Notice also that Example \ref{n8} is the only ``$(n,n)$-bipartite'' linear space to yield a nonunimodal $h$-vector: indeed, it is easy to see that, for any $n\neq 8$, a linear space consisting of two lines with $n$ points each plus all other possible lines containing only 2 points yields a unimodal $h$-vector. It is however possible to construct other interesting nonunimodal ``$(m,n)$-bipartite'' examples when $m\neq n$: for instance, $(m,n)=(7,8)$ gives the $h$-vector
$$(1,15,105,91,105,77,35,9,1).$$
\end{rem}


\end{document}